\documentclass{amsart}
\usepackage[foot]{amsaddr}
\usepackage{amssymb,amsmath,amsfonts,mathptmx,cite,mathrsfs,graphicx,gastex,rotating,url,color}
\usepackage{eucal}

\DeclareMathOperator{\alf}{alph}

\DeclareMathOperator{\dom}{dom}
\DeclareMathOperator{\var}{var}

\theoremstyle{plain}
\newtheorem{theorem}{Theorem}
\newtheorem{proposition}[theorem]{Proposition}
\newtheorem{lemma}[theorem]{Lemma}
\newtheorem{corollary}[theorem]{Corollary}

\theoremstyle{remark}
\newtheorem{remark}[theorem]{Remark}

\DeclareSymbolFont{rsfscript}{OMS}{rsfs}{m}{n}
\DeclareSymbolFontAlphabet{\mathrsfs}{rsfscript}

\def\ib{identity basis}

\newcommand{\fb}{finitely based}
\newcommand{\nfb}{non\-finitely based}
\newcommand{\fs}{finite structure}
\newcommand{\sgp}{semi\-group}
\newcommand{\sgps}{semi\-groups}
\newcommand{\fss}{finite \sgps}
\newcommand{\fbp}{finite basis problem}
\newcommand{\infb}{inherently non\-finitely based}
\newcommand{\snfb}{strongly non\-finitely based}
\newcommand{\fg}{finitely generated}
\newcommand{\pat}{partial transformation}
\newcommand{\pts}{partial transformations}

\newtheorem*{Problem}{Problem}
\newtheorem*{Problem44}{Problem 4.4}

\makeatletter

\renewcommand*\subjclass[2][2010]{\def\@subjclass{#2}\@ifundefined{subjclassname@#1}{\ClassWarning{\@classname}{Unknown edition (#1) of Mathematics Subject Classification; using '2010'.}}{\@xp\let\@xp\subjclassname\csname subjclassname@#1\endcsname}}

\renewcommand{\subjclassname}{\textup{2010} Mathematics Subject Classification}

\makeatother

\begin{document}

\title{Strongly nonfinitely based monoids}
\author[S. V. Gusev]{Sergey V. Gusev}
\address[S. V. Gusev, M. V. Volkov]{Institute of Natural Sciences and Mathematics\\
Ural Federal University\\ 620000 Ekaterinburg, Russia}
\email{sergey.gusb@gmail.com}
\email{olga.sapir@gmail.com}
\email{m.v.volkov@urfu.ru}

\author[O. B. Sapir]{Olga B. Sapir}
\address[O. B. Sapir]{Ben-Gurion University of the Negev, Beer Sheva, Israel}

\author[M. V. Volkov]{Mikhail V. Volkov}

\thanks{S. V. Gusev and M. V. Volkov were supported by the Ministry of Science and Higher Education of the Russian Federation, project FEUZ-2023-2022.}

\begin{abstract}
We show that the 42-element monoid of all partial order preserving and extensive injections on the 4-element chain is not contained in any variety generated by a finitely based finite semigroup.
\end{abstract}



\maketitle

\section{General Background: Identities and the Finite Basis Problem}
\label{sec:intro1}

The idea of an \emph{identity} or a \emph{law} is very basic and is arguably one of~the very first abstract ideas that students come across when they start learning mathematics. We mean laws like the \emph{commutative law of
addition}:
\begin{center}
A sum isn't changed at rearrangement of its addends.
\end{center}
At the end of the high school, a student is aware (or, at least, is supposed to be aware) of a good dozen of laws:
\begin{itemize}
\item[-] the commutative and associative laws of addition,
\item[-] the commutative and associative laws of multiplication,
\item[-] the distributive law of multiplication over addition,
\item[-] the difference of two squares identity,
\item[-] the Pythagorean trigonometric identity,
\end{itemize}
etc, etc. Moreover, the student may feel (though probably cannot explain) the difference between `primary' identities such as
\begin{equation}
ab=ba \label{commutativity}
\end{equation}
and
\begin{equation}
(ab)c=a(bc) \label{associativity}
\end{equation}
and `secondary' ones such as, for instance,
\begin{equation}
(ab)^2=a^2b^2. \label{example}
\end{equation}
`Primary' laws such as \eqref{commutativity} or \eqref{associativity} are \emph{intrinsic} properties of objects (say, numbers) we multiply and of the way the multiplication is defined, whereas `secondary' identities can be \emph{formally inferred} from `primary' ones, without knowing which objects are multiplied and how the multiplication is defined. Here is a simple example of such a formal inference:
\begin{align*}
(ab)^2&=(ab)(ab)&&\text{by the definition of squaring}\\
      &=a(ba)b  &&\text{by the law \eqref{associativity}}\\
      &=a(ab)b  &&\text{by the law \eqref{commutativity}}\\
      &=(aa)(bb)&&\text{by the law \eqref{associativity}}\\
      &=a^2b^2  &&\text{by the definition of squaring}
\end{align*}
Thus, \eqref{example} is a formal corollary of \eqref{associativity} and \eqref{commutativity} and holds whenever and wherever the two laws hold. That is why, when extending the set of natural numbers (positive integers) to the set of integers, and then to the set of rationals, and then to the set of reals, and then to the set of complex numbers, we have to care of preserving \eqref{associativity} and \eqref{commutativity} in the sense that it has to be proved that the laws persist under each of these extensions. In contrast, there is no need to bother with `secondary' identities like \eqref{example} as their formal proofs carry over.

A big part of algebra in fact deals with inferring some useful `secondary identities' from some `primary' laws. Identities to be inferred may be quite complicated, and the inference itself may be highly non-trivial. Think,
for instance, of the product rule for determinant:
\begin{equation}
\det AB=\det A\det B.\label{product rule}
\end{equation}
It looks quite innocent due to convenient notation, but the reader certainly realizes that in fact \eqref{product rule} constitutes a powerful identity whose explicit form is rather bulky already for matrices of a modest size. Indeed, if, say, $A=\begin{pmatrix} a&b\\ c& d\end{pmatrix}$ and $B=\begin{pmatrix} x& y\\ z& t\end{pmatrix}$, then \eqref{product rule} amounts to the identity
\[
(ax+bz)(cy+dt)-(ay+bd)(cx+dz)=(ad-bc)(xt-yz),
\]
and even imagining the explicit form of \eqref{product rule} for $3\times3$-matrices is painful, to say nothing of actually writing it down.

However, one can observe that usually only a few `primary' laws are invoked in the course of the inference even if it is cumbersome. For instance, to deduce the identity \eqref{product rule}, one needs only the very basic laws, namely, the commutative and associative laws of addition and multiplication, the distributive law of multiplication over addition, and the existence of subtraction (that is expressed by the law $a=(a-b)+b$). This observation leads to the idea of composing a \emph{complete} list of~`primary' laws that would allow one to infer \emph{every} possible identity. Such a list is called an \emph{identity basis}. It should be mentioned that even though this usage of the word `basis' is quite common, its meaning here differs from the standard meaning of this term in linear algebra since no independence assumptions are made: the only requirement for a collection of identities $\Sigma$ to form an identity basis is that every identity should be deducible from $\Sigma$!

Of course, in order to speak about an identity basis, one has to specify which identities are under consideration. In this paper, we deal with the simplest non-trivial case of a single \emph{binary} operation. The attribute `binary' means that the operation involves two operands, like addition and multiplication of numbers do. Thus, a binary operation on a non-empty set $S$ is merely a map $S\times S\to S$.

The principal question on which studies of identity bases are focused is known as the \emph{Finite Basis Problem} (FBP, for short). For the purpose of this paper, the FBP may be formulated as follows:
\begin{Problem}[The Finite Basis Problem]
Given a structure $(S,\cdot)$ where $\cdot$ is a binary operation on a set $S$, determine whether or not the identities of $(S,\cdot)$ have a finite basis.
\end{Problem}

The FBP is natural by itself, but it has also revealed a number of interesting and unexpected relations to many issues of theoretical and practical importance ranging from feasible algorithms for membership in certain classes of formal languages to classical number-theoretic conjectures such as the Twin Prime, Goldbach, existence of odd perfect numbers and the infinitude of even perfect numbers---it has been shown by Peter Perkins~\cite{Perkins:1989}
that each of these conjectures is equivalent to the FBP for a structure of the form $(S,\cdot)$.

We say that a structure $(S,\cdot)$ is \emph{\fb} if the answer to the FBP for $(S,\cdot)$ is positive, that is, if the identities of $(S,\cdot)$ have a finite basis. Otherwise,  $(S,\cdot)$ is called \emph{\nfb}.

Even a \emph{finite} structure of the form $(S,\cdot)$ can be \nfb. The smallest example is a 3-element structure known as Murski\v{\i}'s groupoid~\cite{Murskii:1965}. However, arguably, the most striking example (known as the 6-element \emph{Brandt monoid} $B_2^1$) is formed by the following six $2\times2$-matrices:
\begin{equation}\label{eq:b21}
\begin{pmatrix}
1 & 0\\ 0 & 1
\end{pmatrix},\
\begin{pmatrix}
1 & 0\\ 0 & 0
\end{pmatrix},\
\begin{pmatrix}
0 & 1\\ 0 & 0
\end{pmatrix},\
\begin{pmatrix}
0 & 0\\ 1 & 0
\end{pmatrix},\
\begin{pmatrix}
0 & 0\\ 0 & 1
\end{pmatrix},\
\begin{pmatrix}
0 & 0\\ 0 & 0
\end{pmatrix},
\end{equation}
the operation being the usual matrix multiplication. This example is due to Perkins~\cite{Perkins:1969}. Thus, here we see a very transparent, very natural, and very finite structure whose identities cannot be axiomatized by finite means.

In the 1960s, Alfred Tarski~\cite{Ta68} suggested to study the FBP for finite structures as a \emph{decision problem}. Indeed, since any \fs\ is an object that can be given in a constructive way, one can ask for an algorithm which when presented with an effective description of the structure, would determine whether or not it is \fb.

\begin{Problem}[Tarski's Finite Basis Problem]
Is there an algorithm that when given an effective description of a~\fs\ decides whether it is \fb\ or not?
\end{Problem}

This fundamental question was answered in the negative by Ralph McKenzie~\cite{Mc96} who showed that no algorithm can decide the FBP for finite structures of the form $(S,\cdot)$. Thus, no mechanical procedure for answering to the FBP exists in general, and one should be more clever than a computer to get an answer!

\section{The Finite Basis Problem for semigroups and our contribution}
\label{sec:intro2}

In this paper, we deal with the FBP for \emph{semigroups}, that is, structures of the form $(S,\cdot)$ satisfying the associative law \eqref{associativity}. Perkins's example cited in Section~\ref{sec:intro1} revealed that finite semigroups can be \nfb. Moreover, it turns out that \sgps\ are the only `classical' algebras for which finite \nfb\ objects can exist: finite groups \cite{OaPo64}, finite associative and Lie rings \cite{Kr73,Lv73,BaOl75}, finite lattices \cite{McK70} are all \fb. Therefore studying \fss\ from the viewpoint of the FBP has become a hot area in which many neat results have been achieved and several powerful methods have been developed, see the survey \cite{Volkov-01} for an overview. The present paper develops a novel approach to the Finite Basis Problem for finite semigroups initiated in~\cite{Sapir-Volkov-22} and solves one of the problems posed in~\cite{Volkov-01}. As an application, we answer a question left open in~\cite{VoGo03}.

In order to describe our contribution in precise way, we proceed with introducing a few notions and setting up our notation. The basic concepts we need come from equational logic; see, e.g., \cite[Chapter~II]{BuSa81}. For the reader's convenience, we present them here in a form adapted to the use in this paper, that is, specialized to semigroups. When doing so, we closely follow~\cite[Section 1]{Sapir-Volkov-22}.

A (\emph{semigroup}) \emph{word} is a finite sequence of symbols, called \emph{variables}. Sometimes we employ the \emph{empty word}, that is, the empty sequence. Whenever words under consideration are allowed to be empty, we always say it explicitly.

We denote words by lowercase boldface letters. If $\mathbf{w}=x_1\cdots x_k$, where $x_1,\dots,x_k$ are variables, then the set $\{x_1,\dots,x_k\}$ is denoted by $\alf(\mathbf{w})$. If $\mathbf{w}$ is empty, then $\alf(\mathbf{w})=\varnothing$.

Words are multiplied by concatenation, that is, for any words $\mathbf{w}'$ and $\mathbf{w}$, the sequence $\mathbf{ww}'$ is obtained by appending the sequence $\mathbf{w}'$ to the sequence $\mathbf{w}$.

Any map $\varphi\colon\alf(\mathbf{w})\to S$, where $S$ is a semigroup, is called a \emph{substitution}. The \emph{value} $\varphi(\mathbf{w})$ of $\mathbf{w}$ under $\varphi$ is the element of $S$ that results from substituting $\varphi(x)$ for each variable $x\in\alf(\mathbf{w})$ and computing the product in $S$.

A (\emph{semigroup}) \emph{identity} is a pair of words written as a formal equality. From now on, we use the sign $\approx$ when writing identities (so that a pair $(\mathbf{w},\mathbf{w}')$, say, is written as $\mathbf{w}\approx\mathbf{w}'$), saving the standard sign $=$ for `genuine' equalities. A semigroup $S$ \emph{satisfies} $\mathbf{w}\approx \mathbf{w}'$ (or $\mathbf{w}\approx \mathbf{w}'$ \emph{holds} in $S$) if $\varphi(\mathbf{w})=\varphi(\mathbf{w}')$ for every substitution $\varphi\colon\alf(\mathbf{ww}')\to S$, that is, substitutions of~elements from $S$ for the variables occurring in $\mathbf{w}$ or $\mathbf{w}'$ yield equal values to these words.

In Section~\ref{sec:intro1} we mentioned formal inference of identities. For semigroup identities, the inference rules are very transparent as they amount to substituting a word for each occurrence of a variable in an identity, multiplying an identity through on the right or the left by a word, and using symmetry and transitivity of equality. Birkhoff's completeness theorem of equational logic \cite[Theorem 14.17]{BuSa81} gives a clear semantic meaning
to formal inference: an identity $\mathbf{w}\approx\mathbf{w}'$ can be inferred from a set $\Sigma$ of identities if and only if every semigroup satisfying all identities in $\Sigma$ satisfies the identity $\mathbf{w}\approx\mathbf{w}'$ as well. In this situation, we say that an identity $\mathbf{w}\approx\mathbf{w}'$ \emph{follows} from $\Sigma$ or that $\Sigma$ \emph{implies} $\mathbf{w}\approx\mathbf{w}'$.

As defined in Section~\ref{sec:intro1}, a semigroup ${S}$ is \emph{\fb} if it possesses a finite \ib\ and \emph{\nfb} otherwise. We mentioned at the start of this section that the FBP restricted to \fss\ becomes nontrivial; moreover, its algorithmic version, that is, Tarski's Finite Basis Problem restricted to semigroups, remains open so far.

The class of all semigroups satisfying all identities from a given set $\Sigma$ is called the \emph{variety defined by $\Sigma$}. A variety is \emph{\fb} if it can be defined by a finite set of identities; otherwise it is \emph{\nfb}. Given a semigroup ${S}$, the variety defined by  the set of all identities $S$ satisfies is denoted by $\var S$ and called the \emph{variety generated by $S$}. A variety is called \emph{\fg} if it can be generated by a finite semigroup.

A variety is \emph{locally finite} if each of its finitely generated members is finite. A  finite semigroup is called \emph{inherently \nfb} if it is not contained in any finitely based locally finite variety. The very first example of an inherently \nfb\ semigroup was discovered by Mark Sapir~\cite{Sa87a} who proved that the 6-element Brandt monoid $B_2^1$ is inherently \nfb. In~\cite{Sa87b} he gave a structural characterization of all inherently \nfb\ semigroups, which, in particular, led to an algorithm to recognize whether or not a given finite semigroup is inherently \nfb. (This sharply contrasts McKenzie's result \cite{Mc96} that no such algorithm exists for general finite structures.)

It is easy to see that the satisfaction of an identity is inherited by forming direct products and taking \emph{divisors} (that is, homomorphic images of subsemigroups) of semigroups so that each variety is closed under these two operators. In fact, this closure property characterizes varieties (the HSP-theorem; see \cite[Theorem 11.9]{BuSa81}). An easy byproduct of the proof of the HSP-theorem (see \cite[Theorem 10.16]{BuSa81}) is that every \fg\ variety is locally finite. By the definition, a semigroup and the variety it generates are simultaneously finitely or \nfb. Hence, to prove that a given finite semigroup $S$ is \nfb, it suffices to exhibit an inherently \nfb\ semigroup in the variety $\var S$. This argument, combined with Sapir's characterization of all inherently \nfb\ semigroups, has become one of the most powerful and easy-to-use methods in studying the FBP for finite semigroups.

Now let us quote from the survey~\cite{Volkov-01}.
\begin{quote}
If one focuses on the \fbp\ for \fss\ (like we do in this survey), then the notion of an \infb\ \sgp\ appears to be rather abundant. Why should we care about locally finite varieties which are not finitely generated when we are only interested in \fg\ ones? This question leads us to introduce the following notion: call a finite semigroup $S$ \emph{\snfb} if $S$ cannot be a member of any \fb\ \fg\ variety. Clearly, every \infb\ finite semigroup is \snfb, and the question if the converse is true is another intriguing open problem:
\begin{Problem44}
\label{snfb vs infb}
Is there a \snfb\ finite semigroup which is not \infb?
\end{Problem44}
\end{quote}

In this paper, we answer the question asked in~\cite[Problem 4.4]{Volkov-01} in the affirmative. Our example is the 42-element semigroup $IC_4$ from~\cite{Sapir-Volkov-22} where it was shown to have a weaker property. We recall  the definition of the semigroup $IC_4$ and one of its features in Section~\ref{sec:plelim} and then prove our main result in Section~\ref{sec:main}. Section~\ref{sec:application} presents an application.

\section{Preliminaries}
\label{sec:plelim}

Following~\cite[Section 2]{Sapir-Volkov-22}, we introduce the semigroup $IC_4$ as a member of a family of transformation monoids.

Let $[m]$ stand for the set of the first $m$ positive integers ordered in the usual way: $1<2<\dots<m$. By a \emph{\pat} of $[m]$ we mean an arbitrary map $\alpha$ from a subset of $[m]$ (called the \emph{domain} of $\alpha$ and denoted $\dom\alpha$) to $[m]$. We write \pts\ on the right of their arguments. A \pat\ $\alpha$ is \emph{order preserving} if $i\le j$ implies $i\alpha\le j\alpha$ for all $i,j\in\dom\alpha$, and \emph{extensive} if $i\le i\alpha$ for every $i\in\dom\alpha$.  Clearly, if two transformations have either of the properties of being injective, order preserving, or extensive, then so does their product, and the identity transformation enjoys all three properties. Hence, the set of all partial injections of $[m]$ that are extensive and order preserving forms a monoid\footnote{Recall that a \emph{monoid} is a semigroup with an identity element.} that we denote by $IC_m$ and call the $m$th $i$-\emph{Catalan monoid}. Both `I' in the notation and `$i$' in the name mean `injective'; the `Catalan' part of the name again refers to the cardinality of the monoid: $|IC_m|$ is the $(m+1)$-th Catalan number. In particular, $|IC_4|$ is the 5th Catalan number 42 aka the Answer to the Ultimate Question of Life, The Universe, and Everything; see \cite{Adams:1979}.

The key property of the monoid $IC_4$ for this paper involves two combinatorial notions, which we now recall.

Let $\mathbf{u}$ be a word and $x$ a variable in $\alf(\mathbf{u})$. If $x$ occurs exactly once in $\mathbf{u}$, then the variable is called \emph{linear} in $\mathbf{u}$. If $x$ occurs more than once in $\mathbf{u}$, then we say that the variable is \emph{repeated} in $\mathbf{u}$. A word $\mathbf{u}$ is called \emph{sparse} if every two occurrences of a repeated variable in $\mathbf{u}$ sandwich some linear variable.

Given a semigroup $S$, a word $\mathbf{u}$ is called an \emph{isoterm for} $S$ if the only word $\mathbf{v}$ such that $S$ satisfies the identity $\mathbf{u} \approx  \mathbf{v}$ is the word $\mathbf{u}$ itself.

\begin{lemma}[\!{\cite[Lemma~3.4]{Sapir-Volkov-22}}]
\label{lem:sparseC5}
Every sparse word is an isoterm for the monoid $IC_4$.
\end{lemma}

We also need some properties of a class of finite semigroups defined in terms of the Green relation $\mathrsfs D$. Recall that for a semigroup $S$, the notation $S^1$ stands for the least monoid containing $S$, that is, $S^1:=S$ if $S$ has an identity element and  $S^1:=S\cup\{1\}$ if $S$ has no identity element; in the latter case the multiplication in $S$ is extended to $S^1$ in a unique way such that the fresh symbol $1$ becomes the identity element in $S^1$. James Alexander Green (cf. \cite{Gre51}) introduced five equivalence relations on every semigroup $S$ which are collectively referred to as \emph{Green's relations}. Of those five relations, we need the following four:
\begin{itemize}
\item[] $x\,\mathrsfs{R}\,y \Leftrightarrow xS^1 = yS^1$, i.e., $x$ and $y$ generate the same right ideal;
\item[] $x\,\mathrsfs{L}\,y \Leftrightarrow S^1x = S^1y$, i.e., $x$ and $y$ generate the same left ideal;
\item[] $x\,\mathrsfs{J}\,y \Leftrightarrow S^1xS^1 = S^1yS^1$, i.e., $x$ and $y$ generate the same ideal;
\item[] $x\,\mathrsfs{D}\,y \Leftrightarrow (\exists z\in S)\ x\,\mathrsfs{R}\,z \land z\,\mathrsfs{L}\,y$, i.e., $\mathrsfs{D} = \mathrsfs{RL}$.
\end{itemize}
In addition, we write $x\le_{\mathrsfs{J}}y$ if $x\in S^1yS^1$.

An element $e$ of a semigroup $S$ is called an \emph{idempotent} if $e^2=e$. We let $\mathbf{DS}$ stand for the class of all finite semigroups in which every $\mathrsfs D$-class containing an idempotent is a subsemigroup. It is well-known (and easy to verify) that $\mathbf{DS}$ is a \emph{pseudovariety}, that is, a class of finite semigroups closed under forming finite direct products and taking divisors.

The following proposition summarizes the features of semigroups in $\mathbf{DS}$ that we employ. They all can be found (or readily follow from some results) in either Jorge Almeida's monograph \cite{Almeida-95}, where the pseudovariety $\mathbf{DS}$ is comprehensively studied in Chapter 8, or Lev Shevrin's memoir \cite{Shevrin:94}, where Section 3 treats a semigroup class whose finite members exactly constitute $\mathbf{DS}$.

\begin{proposition}
\label{P:DS} Let $S$ be a semigroup in $\mathbf{DS}$.\\
\emph{(a)} Every $\mathrsfs D$-class of $S$ containing an idempotent is a union of its subgroups.\\
\emph{(b)} If $\mathbf u$ and $\mathbf v$ are words with $\alf(\mathbf u)=\alf(\mathbf v)$, then for any substitution $\varphi\colon\alf(\mathbf{u})\to S$ such that $\varphi(\mathbf u)$ is an idempotent, $\varphi(\mathbf u)\le_{\mathrsfs{J}}\varphi(\mathbf v)$.\\
\emph{(c)} If $e\le_{\mathrsfs{J}}a$ and $e\le_{\mathrsfs{J}}b$ for some idempotent $e\in S$ and some $a,b\in S$, then $aeb\,\mathrsfs{D}\,e$.
\end{proposition}

\begin{proof}
Claim (a) is contained in \cite[Theorem 3]{Shevrin:94}; see conditions (4a) or (4c) there.

For (b), we use condition (1b) in \cite[Theorem 3]{Shevrin:94}. It provides a homomorphism $\psi$ from $S$ onto a commutative semigroup of idempotents such that for every idempotent $e$ and every element $a$ in $S$, the equality $\psi(e)=\psi(a)$ implies $e\le_{\mathrsfs{J}}a$. (In terminology of~\cite{Shevrin:94}, this fact is expressed by saying that $S$ is a semilattice of Archimedean semigroups.) The condition $\alf(\mathbf u)=\alf(\mathbf v)$ readily implies $\psi(\varphi(\mathbf u))=\psi(\varphi(\mathbf v))=\prod\limits_{x\in\alf(\mathbf u)}\psi(\varphi(x))$ due to commutativity and idempotency of the semigroup $\psi(S)$. Hence,  $\varphi(\mathbf u)\le_{\mathrsfs{J}}\varphi(\mathbf v)$.

Claim (c) follows from~\cite[Lemma~8.1.4]{Almeida-95} combined with the observation that $\mathrsfs D=\mathrsfs J$ on every finite semigroup~\cite[Theorem 3]{Gre51}.
\end{proof}

The proof of the next lemma closely follows the pattern of the proof of \cite[Lemma~8.1.9]{Almeida-95} but is included for the sake of completeness.
\begin{lemma}\label{L:S-in-DS}
Let $S\in\mathbf{DS}$ and $k:=|S|!$. Then for every word $\mathbf u$ that can be decomposed as $\mathbf u=\mathbf u_0\mathbf u_1\cdots \mathbf u_n$ with $n>|S|$ and $\alf(\mathbf u_0)=\alf(\mathbf u_1)=\cdots=\alf(\mathbf u_n)$, the identity $\mathbf u\approx\mathbf u^{k+1}$ holds in $S$.
\end{lemma}

\begin{proof}
Let $\mathbf w_i:=\mathbf u_0\mathbf u_1\cdots\mathbf u_i$. Take an arbitrary substitution $\varphi\colon\alf(\mathbf{u})\to S$. For brevity, let $u_i:=\varphi(\mathbf u_i)$ and $w_i:=\varphi(\mathbf w_i)$. The $n$ elements $w_0,w_1,\dots,w_{n-1}$ may not be all distinct, and so there exist indices $p,q$ with $0\le p<q<n$ such that $w_p=w_q$. Hence
\[
w_q=w_pu_{p+1}u_{p+2}\cdots u_q=w_qu_{p+1}u_{p+2}\cdots u_q,
\]
from which we deduce the equality
\begin{equation}\label{eq:repetition}
w_q=w_q(u_{p+1}u_{p+2}\cdots u_q)^k.
\end{equation}
It is known (and easy to verify) that the $k$th power of any element of $S$ is an idempotent. Since $\alf(\mathbf w_q)=\alf(\mathbf u_i)=\alf(\mathbf u_{q+1}\mathbf u_{q+2}\cdots\mathbf u_n)$ for $i=0,1,\dots,n$,  Proposition~\ref{P:DS}(b) implies that $(u_{p+1}u_{p+2}\cdots u_q)^k\le_{\mathrsfs{J}}w_q$ and $(u_{p+1}u_{p+2}\cdots u_q)^k\le_{\mathrsfs{J}}u_{q+1}u_{q+2}\cdots u_n$. Then by Proposition~\ref{P:DS}(c) the element
\[
w_n=w_qu_{q+1}u_{q+2}\cdots u_n\stackrel{\eqref{eq:repetition}}{=}w_q(u_{p+1}u_{p+2}\cdots u_q)^ku_{q+1}u_{q+2}\cdots u_n
\]
and the idempotent $(u_{p+1}u_{p+2}\cdots u_q)^k$ lie in the same $\mathrsfs{D}$-class. By Proposition~\ref{P:DS}(a) the $\mathrsfs D$-class of the element $w_n$ is a union of its subgroups. Thus, $w_n$ belongs to a subgroup of $S$. Then the idempotent $w_n^k$ is the identity element of this subgroup, and $w_n^{k+1}=w_n$. Consequently, we have
\[
\varphi(\mathbf u)=w_n=w_n^{k+1}=\varphi(\mathbf u^{k+1}).
\]
Since the substitution $\varphi\colon\alf(\mathbf{u})\to S$ is arbitrary, $S$ satisfies the identity $\mathbf u\approx \mathbf u^{k+1}$.
\end{proof}

By $B_2$ we denote the subsemigroup of the Brandt monoid $B_2^1$ consisting of the five non-identity matrices in~\eqref{eq:b21}. The following characterization of finite semigroups beyond $\mathbf{DS}$ occurs as Exercise~8.1.6 in \cite{Almeida-95}; the solution to this exercise follows from \cite[Theorem 3]{Shevrin:94}.

\begin{lemma}
\label{lem:ds}
A finite semigroup $S$ does not belong to the pseudovariety\/ $\mathbf{DS}$ if and only if $S\times S$ has the semigroup $B_2$ as a divisor.
\end{lemma}
For each idempotent $e$ of a semigroup $S$, the set $eSe:=\{ese\mid s\in S\}$ is a subsemigroup in which $e$ serves as an identity element. We call $eSe$ the \emph{local submonoid of $S$ at} $e$. By $\mathbf{LDS}$ we denote the class of all finite semigroups all of whose local submonoids lie in $\mathbf{DS}$. The class $\mathbf{LDS}$ also forms a pseudovariety; see \cite[Section 5.2]{Almeida-95}. We need the following corollary of Lemma~\ref{lem:ds}.

\begin{corollary}
\label{cor:lds}
A finite semigroup $S$ does not belong to the pseudovariety\/ $\mathbf{LDS}$ if and only if $S\times S$ has the monoid $B^1_2$ as a divisor.
\end{corollary}

\begin{proof}
For the `if' part, observe that $B^1_2\notin\mathbf{LDS}$. Indeed, $B^1_2$ is a local submonoid of itself, and the four matrix units in \eqref{eq:b21} form a $\mathrsfs D$-class that contains an idempotent matrix but is not closed under matrix multiplication. Now the claim follows from $\mathbf{LDS}$ being closed under forming finite direct products and taking divisors.

For the `only if' part, take an arbitrary finite semigroup $S\notin\mathbf{LDS}$. Then for some idempotent $e\in S$, the local submonoid $eSe$ does not belong to the pseudovariety $\mathbf{DS}$. By Lemma~\ref{lem:ds} we conclude that the monoid $T:=eSe\times eSe$ has the semigroup $B_2$ as a divisor. Consider a subsemigroup $U$ of $T$ such that there exists an onto homomorphism $\varphi\colon U\to B_2$. The identity element $f:=(e,e)$ of $T$ cannot belong to $U$ since otherwise its image $\varphi(f)$ would be an identity element in $B_2$, and $B_2$ has no identity element. The union $U'=U\cup\{f\}$ is a subsemigroup of $T$. We extend the homomorphism $\varphi$ to an onto map $\varphi'\colon U'\to B_2^1$, letting $\varphi'(f):=\begin{pmatrix}1 & 0\\ 0 & 1\end{pmatrix}$. Clearly, $\varphi'$ is a homomorphism whence  the monoid $B^1_2$ as a divisor of $T$ which is a submonoid of $S\times S$.
\end{proof}

\section{Main result}
\label{sec:main}

The paper~\cite{Sapir-Volkov-22} has promoted the idea of relativizing the property of being inherently \nfb\ (first suggested in~\cite{JaVo09} in the context of quasivarieties). If $\mathbf C$ is a class of semigroups, a  semigroup $T$ is called \emph{inherently \nfb\ relative to} $\mathbf C$ if every semigroup $S\in\mathbf C$ such that $T\in\var S$ is \nfb. Specializing $\mathbf C$, one gets various concepts that occur in the literature. For instance, the property of being inherently \nfb\ as considered by Mark Sapir in~\cite{Sa87a,Sa87b} arises when $\mathbf C$ consists of all semigroups that generate locally finite varieties. If $\mathbf C$ is the class of all finite semigroups, one gets the property of being \snfb\ discussed in Section~\ref{sec:intro2}.

Theorem~3.1 in~\cite{Sapir-Volkov-22} shows that the $i$-Catalan monoid $IC_4$ is inherently nonfinitely based relative to the class of all finite semigroups in which Green's relation $\mathrsfs{R}$ is trivial (that is, coincides with the equality relation). We strengthen this result in Theorem~\ref{T:EDS-infb} below, but first we provide a sufficient condition on a class of semigroups, under which $IC_4$ is inherently nonfinitely based relative to this class.

We fix a countably infinite set $\mathfrak A$ of variables. Denote by $\mathfrak A^+$ the set of all words whose variables lie in $\mathfrak A$ and let $\mathfrak A^*$ be $\mathfrak A^+$ with the empty word added. We assume that all words that we encounter below come from $\mathfrak A^+$.

Let $\mathbf{w}$ be a word. For $X\subseteq\alf(\mathbf{w})$, we denote by $\mathbf{w}(X)$ the word obtained from $\mathbf{w}$ by removing all occurrences of variables from $\alf(\mathbf{w})\setminus X$. An occurrence of a word $\mathbf{u}$ in a word $\mathbf{w}$ as a \emph{factor} is any decomposition of the form $\mathbf{w}=\mathbf{v}'\mathbf{u}\mathbf{v}''$ where the words $\mathbf{v}',\mathbf{v}''$ may be empty. If such a decomposition of $\mathbf{w}$ is unique, then we say that the factor $\mathbf{u}$ occurs in $\mathbf{w}$ once; otherwise, $\mathbf{u}$ occurs in $\mathbf{w}$ more than once.

\begin{proposition}
\label{P:general-infb}
Suppose that $\mathbf C$ is a class of semigroups and for each semigroup $S\in\mathbf C$ such that the $i$-Catalan monoid $IC_4$ belongs to the variety $\var S$, there exist an infinite sequence $\{\mathbf u_n\approx \mathbf v_n\}$ of identities holding in $S$ and an infinite sequence $\{X_n\}$ of sets of variables such that
\begin{itemize}
\item[\textup{(P0)}] $\mathbf u_n(X_n)\ne \mathbf v_n(X_n)$;
\item[\textup{(P1)}] for all variables $y,z$, the word $yz$ occurs in $\mathbf u_n(X_n)$ as a factor at most once;
\item[\textup{(P2)}] for every variable $z$, there are at least $n$ pairwise distinct variables between any two occurrences of $z$ in $\mathbf u_n(X_n)$.
\end{itemize}
Then the $i$-Catalan monoid $IC_4$ is inherently nonfinitely based relative to the class\/ $\mathbf C$.
\end{proposition}

\begin{proof}
We have to verify that each semigroup $S\in\mathbf C$ such that $IC_4\in\var S$ is nonfinitely based. For this, it suffices to exhibit a property $\theta$ of words such that
\begin{itemize}
\item[\textup{(i)}] the word $\mathbf u_n$ has the property $\theta$, while the word $\mathbf v_n$ does not have the property $\theta$;
\item[\textup{(ii)}] for an arbitrary identity $\mathbf u_n \approx \mathbf u$ of $S$ such that the word $\mathbf u$ has the property $\theta$, an application of any identity of $S$ in less than $n-2$ variables to the word $\mathbf u$ preserves the property $\theta$.
\end{itemize}
Indeed, a standard syntactic argument (see~\cite[Section~4]{Volkov-01} or~\cite[Fact~2.1]{Sapir-15}) then implies that for each $n$, the identity $\mathbf u_n\approx \mathbf v_n$ cannot be inferred from identities in less
less than $n-2$ variables holding in $S$. Therefore, no finite set of identities holding in $S$ can infer all identities of this semigroup. 

We show that the following property $\theta$ is relevant: a word $\mathbf w$ has $\theta$ if $\mathbf w(X_n)=\mathbf u_n(X_n)$. Evidently,~(i) holds by the property (P0). It remains to verify~(ii) provided that $IC_4\in\var S$.

Let $\mathbf u_n \approx \mathbf u$ be an identity of $S$ such that $\mathbf u(X_n)=\mathbf u_n(X_n)$. We need to establish that if a word $\mathbf v$ is obtained from $\mathbf u$ by an application of some identity $\mathbf s\approx\mathbf t$ of $S$ in less than $n-2$ variables, then $\mathbf v(X_n)=\mathbf u_n(X_n)$. Obtaining $\mathbf v$ from $\mathbf u$ by an application of $\mathbf s\approx\mathbf t$ means that $\mathbf u=\mathbf c\,\varphi(\mathbf s)\,\mathbf d$ and $\mathbf v=\mathbf c\,\varphi(\mathbf t)\,\mathbf d$ for some $\mathbf c,\mathbf d\in \mathfrak A^\ast$ and some substitution $\varphi\colon\alf(\mathbf s\mathbf t)\to \mathfrak A^+$.

Take two variables $c,d\notin\alf(\mathbf s\mathbf t)$. The identity $\mathbf s\approx \mathbf t$ implies each of the identities $c\,\mathbf s\,d\approx c\,\mathbf t\,d$, $c\,\mathbf s\approx c\,\mathbf t$, and $\mathbf s\,d\approx \mathbf t\,d$. If the words $\mathbf c$ and $\mathbf d$ are nonempty, then $\mathbf u=\psi(c\,\mathbf s\,d)$ and $\mathbf v=\psi(c\,\mathbf t\,d)$, where $\psi\colon \alf(\mathbf s\mathbf t\,cd)\to \mathfrak A^+$ is the substitution given by $\psi(c):=\mathbf c$, $\psi(d):=\mathbf d$ and $\psi(x):=\varphi(x)$ for each $x\in\alf(\mathbf s\mathbf t)$. Similarly, if one of the words $\mathbf c$ and $\mathbf d$ is empty while the other is not, then the words $\mathbf u$ and $\mathbf v$ are images of either the words $c\,\mathbf s$ and respectively $c\,\mathbf t$ or the words $\mathbf s\, d$ and respectively $\mathbf t\,d$ under a suitable substitution. It follows that we may assume without any loss that $\mathbf u=\varphi(\mathbf s)$ and $\mathbf v=\varphi(\mathbf t)$, and $\mathbf s\approx\mathbf t$ is an identity of $S$ in less than $n$ variables.

{Let
\[
Y_n:=\{z\in\alf(\mathbf s\mathbf t)\mid \alf(\varphi(z))\cap X_n\ne \varnothing\}.
\]
Let us verify that the word $\mathbf s(Y_n)$ is sparse. Indeed, for every repeated variable $y$ of $\mathbf s$, the word $\varphi(y)$ occurs as a factor in $\mathbf u$ more than once. In view of the property~(P1), we see that for each variable $y\in Y_n$ repeated in $\mathbf s$, the word $\varphi(y)(X_n)$ must be a single variable $x\in X_n$, say. Now choose two occurrences of ${_1}y$ and ${_2}y$ of $y$ in $\mathbf s$ and let ${_1}x$ and ${_2}x$ be the corresponding occurrences of $x$ in $\mathbf u$. By the property~(P2) there are at least $n$ pairwise distinct variables from $X_n$ between ${_1}x$ and ${_2}x$ in $\mathbf u$. Since $|\alf(\mathbf s)| < n$ and $Y_n$ is the set of all variables whose images under $\varphi$ contain variables from $X_n$, there must be a variable $t\in Y_n \cap\alf(\mathbf s)$ such that $\varphi(t)$ involves at least two variables in $X_n$. In view of the property~(P1), the variable $t$ must be linear in $\mathbf s$. Therefore, the word $\mathbf s(Y_n)$ is sparse.

Since $IC_4\in\var S$, the identity $\mathbf s\approx \mathbf t$ holds in $IC_4$. As $IC_4$ is a monoid, so does the identity $\mathbf s(Y_n)\approx \mathbf t(Y_n)$ since removing all occurrences of variables from $\alf(\mathbf{st})\setminus Y_n$ has the same effect as substituting the identity element of $IC_4$ for these variables. By Lemma~\ref{lem:sparseC5} every sparse word is an isoterm for the $i$-Catalan monoid $IC_4$. It follows that $\mathbf t(Y_n)=\mathbf s(Y_n)$. Hence $\mathbf v(X_n)=\mathbf u(X_n)=\mathbf u_n(X_n)$, as required.}
\end{proof}

\begin{theorem}
\label{T:EDS-infb}
The i-Catalan monoid $IC_4$ is inherently nonfinitely based relative to the pseudovariety $\mathbf{LDS}$.
\end{theorem}

\begin{proof}
Take any $S\in\mathbf{LDS}$ such that the variety $\var S$ contains $IC_4$; we have to prove that $S$ is nonfinitely based.

Let $k=|S|!$; then the $k$th power of any element of $S$ is an idempotent. In view of Proposition~\ref{P:general-infb}, it suffices to find an infinite sequence $\{\mathbf u_n\approx \mathbf v_n\}$ of identities holding in $S$ and an infinite sequence $\{X_n\}$ of sets of variables such that the properties~(P0),~(P1) and~(P2) hold. We will show that the following are relevant:
\[
\mathbf u_n:=\prod_{i=0}^n\mathbf a_n[\pi^i]\mathbf b_n[\pi^i],\ \
\mathbf v_n:=(\mathbf u_n)^{k+1},\ \
X_n:=\{x_0,y_0,z_0,x_1,y_1,z_1,\dots,x_n,y_n,z_n\}
\]
where $\pi$ denotes the cyclic permutation $(01\cdots n)$ of the set $\{0,1,\dots,n\}$, and
\[
\begin{aligned}
&\mathbf a_n[\tau]:=x^kx_{0\tau}x^ky_1x^kx_{1\tau}x^ky_2x^kx_{2\tau}x^k\cdots x^ky_nx^kx_{n\tau}x^k,\\
&\mathbf b_n[\tau]:=x^kz_{0\tau}x^ky_1x^kz_{1\tau}x^ky_2x^kz_{2\tau}x^k\cdots x^ky_nx^kz_{n\tau}x^k
\end{aligned}
\]
for any permutation $\tau$ of $\{0,1,\dots,n\}$.

By the definitions of the identities $\mathbf u_n\approx \mathbf v_n$ and the sets $X_n$, the properties (P0), (P1) and (P2) hold for each $n$. Since 
\[
\alf(\mathbf a_n[\pi^0]\mathbf b_n[\pi^0])=\ldots=\alf(\mathbf a_n[\pi^i]\mathbf b_n[\pi^i])=\ldots=\alf(\mathbf a_n[\pi^n]\mathbf b_n[\pi^n]),
\]
Lemma~\ref{L:S-in-DS} implies that every local submonoid of $S$ satisfies the identity $\mathbf u_n(X_n)\approx \mathbf v_n(X_n)$ for all $n>|S|$. Since the $k$th power of any element of $S$ is an idempotent, this implies that the identity $\mathbf u_n\approx \mathbf v_n$ holds in $S$. Theorem~\ref{T:EDS-infb} is proved.
\end{proof}

Now it easy to deduce our main result. Recall that a semigroup is said to be \snfb{} if it is inherently nonfinitely based relative to the class of all finite semigroups.

\begin{theorem}
\label{T:strong}
The i-Catalan monoid $IC_4$ is strongly nonfinitely based.
\end{theorem}

\begin{proof}
Take any finite semigroup $S$ such that the variety $\var S$ contains $IC_4$; we have to prove that $S$ is nonfinitely based. If $S\in\mathbf{LDS}$, this follows from Theorem~\ref{T:EDS-infb}. If $S\notin\mathbf{LDS}$, then Corollary~\ref{cor:lds} implies that the variety $\var S$ contains the 6-element Brandt monoid $B_2^1$. Since $B_2^1$ is inherently \nfb\ \cite[Corollary 6.1]{Sa87a}, we conclude that $S$ is nonfinitely based is this case as well.
\end{proof}

It readily follows from the structural characterization of inherently \nfb\ semigroups \cite[Theorem 1]{Sa87b} that such a semigroup must have a non-singleton $\mathrsfs D$-class. Since all $\mathrsfs D$-classes of the $i$-Catalan monoid $IC_4$ are singletons, we conclude that $IC_4$ is not inherently \nfb. Thus, Theorem~\ref{T:strong} provides an example of a \snfb\ semigroup which is not \infb, answering the question from \cite{Volkov-01} quoted in Section~2.

\begin{remark}
Reviewing the proofs of Proposition~\ref{P:general-infb} and Theorems~\ref{T:EDS-infb} and~\ref{T:strong}, one sees that all our arguments rely on only two properties of $IC_4$: that $IC_4$ is a monoid and that every sparse word is an isoterm for $IC_4$. Therefore, any monoid for which every sparse word is an isoterm is \snfb. Using this, the first-named author has constructed a \snfb\ monoid with only 9 elements which is not \infb. This result will be published separately.
\end{remark}

\begin{remark}
We point out a subtle yet important difference between the concept of being \infb\ as considered in~\cite{Sa87a,Sa87b} and that of being \snfb. The difference comes from the fact that the local finiteness of a variety is
inherited by its subvarieties while the property of being finitely generated is not. Therefore, if a semigroup $S$ is not contained in any \fb\ locally finite semigroup variety, then $S$ is contained in no \fb{} locally finite variety $\mathbf V$ of groupoids---otherwise, the intersection of $\mathbf V$ with the variety of all semigroups would be a \fb\ locally finite variety of semigroups containing $S$. Thus, when we speak about \infb{} semigroups, it is unnecessary to specify within which class we work. In contrast, when we speak about \snfb{} semigroups, we should distinguish between the ``absolute'' case and the case when we work within the class of all semigroups. In the present paper we have only proved that every finitely generated \textbf{semigroup} variety containing the monoid $IC_4$ is \nfb. This does not exclude the possibility that some \fb\ finitely generated \textbf{groupoid} variety contains $IC_4$. The question of whether or not there exists a semigroup which, being not \infb, is \snfb{} relative to the class of all finite groupoids still remains open.

For a more detailed discussion of the property of being \snfb\ in a broader universal-algebraic context, we refer the reader to \cite[Section 1.1]{JaMcN11}. 
\end{remark}

\section{An application}
\label{sec:application}

Theorem~\ref{T:strong} can be applied to prove the absence of a finite \ib\ for many finite semigroups for which the FBP remained open so far. Here we restrict ourselves to just one application, resolving a question left open in~\cite{VoGo03}.

Let $T_n(q)$ stand for the semigroup of all upper triangular $n\times n$-matrices over the finite field with $q$ elements. In~\cite{VoGo03}, it was shown that the semigroup $T_n(q)$ is inherently infinitely based if and only if $q>2$ and $n>3$. Thus, semigroups of upper triangular matrices over the 2-element field turn out to be not \infb, but the question of whether or not they are \fb\ remained unsolved for 20 years, with the only exception of the 8-element semigroup $T_2(2)$ that was proved to be \fb\ in~\cite{ZhLL12}. Now we are in a position to answer the question for all $n>3$.

\begin{theorem}
\label{T:triangular}
For each $n>3$, the semigroup $T_n(2)$ of all upper triangular $n\times n$-matrices over the $2$-element field is \textup(strongly\textup) \nfb.
\end{theorem}

\begin{proof}
Due to Theorem~\ref{T:strong}, it suffices to show that for each $n>3$, the variety $\var T_n(2)$ contains the $i$-Catalan monoid $IC_4$. In fact, we construct an embedding $IC_4\to T_4(2)$; since $T_4(2)$ naturally embeds into $T_n(2)$ for all $n\ge4$, the claim will follow.

Recall that the monoid $IC_4$ consists of all extensive and order preserving partial injections of the chain $1<2<3<4$ into itself. Given any such partial injection $\alpha$, we define a $4\times4$-matrix $A:=(a_{ij})$ over the 2-element field by setting $a_{ij}:=\begin{cases}1&\text{if } i\alpha=j,\\0&\text{otherwise}\end{cases}$. Since $\alpha$ is extensive, $i\alpha=j$ implies $i\le j$ whence the matrix $A$ is upper triangular. Clearly, the map $\alpha\mapsto A$ is one-to-one, and it is easy to verify that the map is a homomorphism, using the fact that the image of $IC_4$ consists of row-monomial matrices so that one never adds two 1s when multiplying such matrices.
\end{proof}


\small
\begin{thebibliography}{99}
\bibitem{Adams:1979}
Adams, D.:  The Hitchhiker's Guide to the Galaxy. Pan Books, London (1979)

\bibitem{Almeida-95}
Almeida, J.: Finite Semigroups and Universal Algebra. Series in Algebra, vol. 3. World Scientific, Singapore (1994)

\bibitem{BaOl75}
Bakhturin, Yu.A., Ol'shanski\v{\i}, A.Yu.: Identical relations in finite Lie rings, Mat. Sb. \textbf{96}, 543--559 (1975) [In Russian; English translation: Mathematics of the USSR--Sbornik \textbf{25}, 507--523 (1975)]

\bibitem{BuSa81}
Burris, S., Sankappanavar, H.P.: A Course in Universal Algebra. Springer, Berlin, Heidelberg, New York (1981)

\bibitem{Gre51}
Green, J.A.: {On the structure of semigroups}, Ann. Math. (2) \textbf{54}, 163--172 (1951)

\bibitem{JaMcN11}
Jackson, M., McNulty, G.F.: The equational complexity of Lyndon's algebra. Algebra Universalis \textbf{65}, 243--262 (2011)

\bibitem{JaVo09}
Jackson, M., Volkov, M.V.: Relatively inherently nonfinitely q-based semigroups. Trans. Amer. Math. Soc. \textbf{361}(4), 2181--2206 (2009)

\bibitem{Kr73}
Kruse, R.L.: Identities satisfied by a finite ring. J. Algebra \textbf{26}, 298--318 (1973)

\bibitem{Lv73}
L'vov, I.V.:  Varieties of associative rings. I. Algebra i Logika \textbf{12}, 269--297 (1973) [In Russian; English translation: Algebra and Logic \textbf{12}, 150--167 (1973)]

\bibitem{McK70}
McKenzie, R.N.: Equational bases for lattice theories, Math. Scand. \textbf{27}, 24--38 (1970)

\bibitem{Mc96}
McKenzie, R.: Tarski's finite basis problem is undecidable. Int. J. Algebra Comput. \textbf{6}, 49--104 (1996)

\bibitem{Murskii:1965}
Murski\v{\i}, V.L.: The existence in three-valued logic of a closed class with finite basis not having a finite complete system of identities. Dokl. Akad. Nauk SSSR \textbf{103}, 816--818 (1965) [In Russian; English translation: Soviet Math. Dokl. \textbf{6}, 1020--1024 (1965)]

\bibitem{OaPo64}
Oates, S., Powell, M.B.: Identical relations in finite groups, J. Algebra \textbf{1}, 11--39 (1964)

\bibitem{Perkins:1969}
Perkins, P.: Bases for equational theories of semigroups. J. Algebra \textbf{11}, 298--314 (1969)

\bibitem{Perkins:1989}
Perkins, P.: Finite axiomatizability for equational theories of computable groupoids. J. Symbolic Logic \textbf{54}, 1018--1022 (1989)

\bibitem{Sa87a}
Sapir, M.V.: {Problems of Burnside type and the finite basis property in varieties of semigroups}. Izv. Akad. Nauk SSSR, Ser. Mat. \textbf{51}, 319--340 (1987) [In Russian; English translation: Mathematics of the USSR--Izv. \textbf{30}, 295--314 (1988)]

\bibitem{Sa87b}
Sapir, M.V.: Inherently nonfinitely based finite semigroups. Mat. Sb. \textbf{133}(2), 154--166 (1987) [In Russian; English translation: Mathematics of the USSR--Sb. \textbf{61}, 155--166 (1988)]

\bibitem{Sapir-15}
Sapir, O.B.: Non-finitely based monoids. Semigroup Forum \textbf{90}, 557--586 (2015)

\bibitem{Sapir-Volkov-22}
Sapir, O.B., Volkov, M.V.: Catalan monoids inherently nonfinitely based relative to finite $R$-trivial semigroups.  J. Algebra \textbf{633}, 138--171 (2023)

\bibitem{Shevrin:94}
Shevrin, L.N.: On the theory of epigroups. I.  Mat. Sb. \textbf{185}(8),  129--160 (1994) [In Russian; English translation:  Russian Acad. Sci. Sb. Math. \textbf{82}(2), 485–-512 (1995)]

\bibitem{Ta68}
Tarski, A.: Equational logic and equational theories of algebras. In: Schnodt, H.A., Sch\"utte, K., Thiele, H.J. (eds.), Contributions to Mathematical Logic: Proc. Logic Colloq., Hannover, 1966, pp. 275--288. North-Holland, Amsterdam (1968)

\bibitem{Volkov-01}
Volkov, M.V.: The finite basis problem for finite semigroups. Sci. Math. Jpn. \textbf{53}, 171--199 (2001)

\bibitem{VoGo03}
Volkov, M.V., Goldberg, I.A.: Identities of semigroups of triangular matrices over finite fields. Mat. Zametki \textbf{73}(4), 502--510 (2003) [In Russian; English translation: Math. Notes \textbf{73}(4), 474--481 (2003)]

\bibitem{ZhLL12}
Zhang, W.T., Li, J.R., Luo, Y.F., On the variety generated by the monoid of triangular $2\times2$ matrices over a two-element field. Bull. Aust. Math. Soc. \textbf{86}(1), 64--77 (2012)
\end{thebibliography}
\end{document}